\documentclass{amsart}
\usepackage{graphicx,xcolor}
\usepackage{amsmath,amssymb,amsthm,amsrefs}
\usepackage{hyperref}
\usepackage{verbatim}
\usepackage{subcaption}
\usepackage{ulem}

\usepackage[margin=3.5cm]{geometry}

\newtheorem{thm}{Theorem}

\newtheorem{lem}[thm]{Lemma}

\theoremstyle{definition}
\newtheorem{defn}[thm]{Definition}
\newtheorem{rmk}[thm]{Remark}

\definecolor{wineRed}{rgb}{0.7,0,0.3}
\newcommand{\PR}[1]{
#1}
\newcommand{\GW}[1]{
#1}
\newcommand{\R}{\mathbb{R}}
\newcommand{\bR}{\mathbb{R}}

\newcommand{\SL}{\mathcal{L}}
\newcommand{\LocSmooth}{C^\infty_{loc}(\R)}

\title{A classification of solitons for the surface diffusion flow of entire graphs}

\author[P.~Rybka]{Piotr Rybka}
\address[P.~Rybka]{University of Warsaw \\ 
Faculty of Mathematics, Informatics and Mechanics \\
ul. Banacha 2\\
02-097 Warsaw, Poland \\ 
ORCiD ID \url{https://orcid.org/0000-0002-0694-8201}}
\email{rybka@mimuw.edu.pl}

\author[G.~Wheeler]{Glen Wheeler}
\address[G.~Wheeler]{Institute for Mathematics and its Applications \\
University of Wollongong\\
Northfields Avenue\\
Wollongong, NSW, 2522, Australia\\
ORCiD ID \url{https://orcid.org/0000-0003-3314-5647}
}
\email{glenw@uow.edu.au}

\begin{document}

\begin{abstract}
In this article we 
classify solitons (equilibria, self-similar solutions and travelling waves) for the surface diffusion flow of entire graphs of function over $\mathbb{R}$.
\end{abstract}



\maketitle

\section{Introduction}

The surface diffusion flow was introduced in a seminal paper of Mullins
\cite{M57}.
There, he studied the formation of thermal grooves in sheets of material.
Motivated by experimentation, Mullins' mathematical formulation assumes that
the sheet is determined by a family of profile functions
$u:\R\times[0,T)\rightarrow\R$ orthogonal to the thermal groove.
That is, there is an assumed translation invariance parallel to the groove in
Mullins' model.
This makes the configuration being modelled two-dimensional, justifying the use of the word `surface'.

The evolution equation as proposed by Mullins is (see \cite{M57}*{(11)})
\begin{equation}
\frac{\partial u}{\partial t} = -\frac{\partial}{\partial x} \bigg(
					\frac{1}{\sqrt{1+(\partial u/\partial x)^2}}
					\frac{\partial}{\partial x}\bigg(
						\frac{\partial^2 u/\partial x^2}{\sqrt{1+(\partial u/\partial x)^2}^3}
					\bigg)
				\bigg)
\label{CDF}
\end{equation}
where we have set the physical constant $D_s\gamma\Omega^2\nu/kT = B>0$, determined by parameters from the specific setting, to 1 (following Mullins).

Since Mullins' paper, a great number of works have appeared studying surface diffusion.
As explained in Cahn-Taylor \cite{CT94} the surface diffusion operator is an important object to study in its own right.
By the time of \cite{CT94} not only were many further physical settings
discovered to be modeled by surface diffusion and its generalisations, but the
hallmark geometric properties of surface diffusion discovered.
Namely, that for an immersed surface, the surface diffusion flow conserves
signed enclosed volume and reduces surface area, with equilibria consisting
precisely of surfaces with constant mean curvature.
Cahn-Taylor argued that a comprehensive mathematical theory for surface diffusion flow needs to be developed.

In terms of existence and uniqueness for solutions to \eqref{CDF}, the best
result to our knowledge is that of Asai \cites{A10,A12}, where it is proved that
unique solutions exist from \emph{bounded} initial data of class
$h^{1+\alpha}$ (the closure of bounded uniformly continuous functions of order
$1+\alpha$ in the space of bounded uniformly smooth functions).
If the initial data is Lipschitz (in $C^{0,1}(\R^n)$) with small Lipschitz constant, then
\cite{KL12} may also be used to generate a unique solution.
In higher dimensions we mention the result \cite{LSS20} which is in the same
regularity class as Asai and remarkably general. More recently, \cite{GGK} offers an extension of this approach to a class of problems where the velocity of the curve is the Laplace-Betrami operator acting on a function of the curvature.

Here, we focus on the study of entire graphical solutions to the flow \eqref{CDF}.
Specifically, we are interested in the classification of solutions moving
according to a symmetry action of the ambient plane, that is, solitons.
As rotations will not preserve graphicality, we  focus on the cases of (a)
self-similar solutions; and (b) travelling waves.

As mentioned in Asai-Giga \cite{AG14} (although there the half-infinite problem
is focused on), linear functions $u(x,t) = Ax$ where $A\in\R$ are solutions to
\eqref{CDF}.
Apart from the case of $A=0$, all of these solutions are unbounded.
Our main result is that these are the \emph{only} graphical solitons under very general conditions.

\begin{thm}
Let $\phi:\R\rightarrow\R$ be a locally smooth function.
Assume that either
\begin{enumerate}
\item[(a)] $\phi$ is a steady state,
\item[(b)] $\phi$ is a 
forward self-similar profile \GW{with $\phi(0) = 0$ (or almost linear (in the sense of \eqref{EQalmoststraight})},
\item[(c)] $\phi$ is a backward self-similar profile and almost linear, or
\item[(d)] $\phi$ is a travelling wave profile
\end{enumerate}
for the surface diffusion flow.
Then $\phi(y) = Ay$ for some $A\in\R$.
\label{TMmain}
\end{thm}

\begin{rmk}
Our theorem \GW{indicates that most} non-trivial solitons are unbounded.
Thus, in order to study the dynamics of the graphical surface diffusion
flow, it is imperative that an 
existence and uniqueness result that allows unbounded initial data be established.
To our knowledge this does not yet exist in the literature. 
\end{rmk}

\begin{rmk}
We have adopted the forward/backward terminology for self-similar profiles used in \cite{GGS10}.
In geometric flows literature these correspond to expanding and shrinking solutions respectively, and travelling waves correspond to translating solutions.
\end{rmk}


For the travelling wave case, our triviality result is sharp in the following sense.
In \cite{KK} the authors construct a family of travelling wave profiles with  fixed contact angles. 
The waves constructed in \cite{KK} are defined on a bounded interval.
If viewed as functions on $\bR$, they are not smooth.
Thus triviality of travelling waves does not hold if either the regularity of the solution is weakened or the domain is allowed to be a bounded interval.


\begin{figure}[t]
\centering
\begin{subfigure}[b]{0.3\textwidth}
    \centering
    \includegraphics[width=\textwidth]{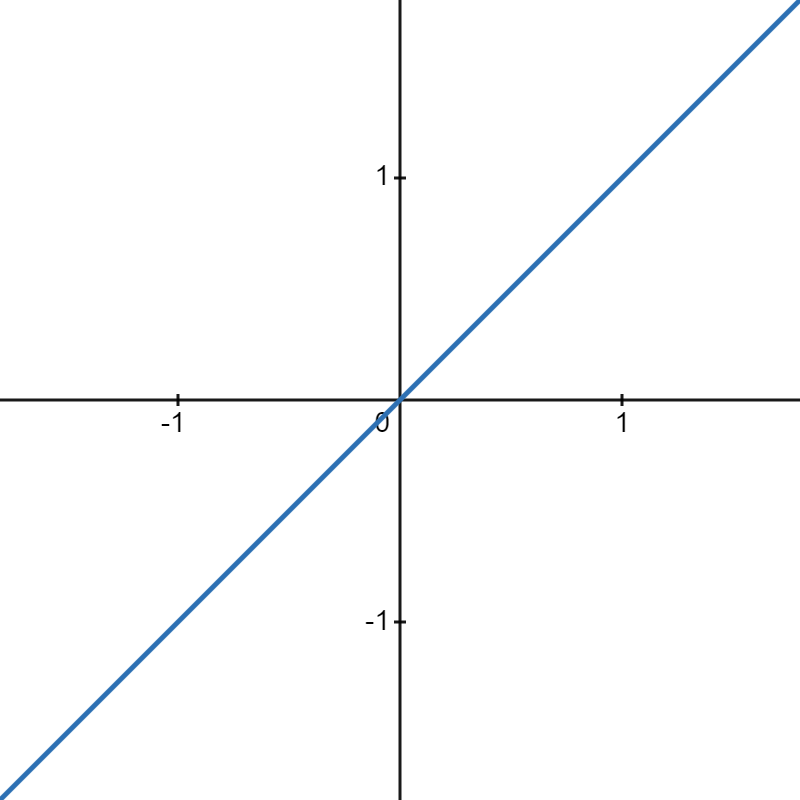}
\end{subfigure}
\hfill 
\begin{subfigure}[b]{0.3\textwidth}
    \centering
    \includegraphics[width=\textwidth]{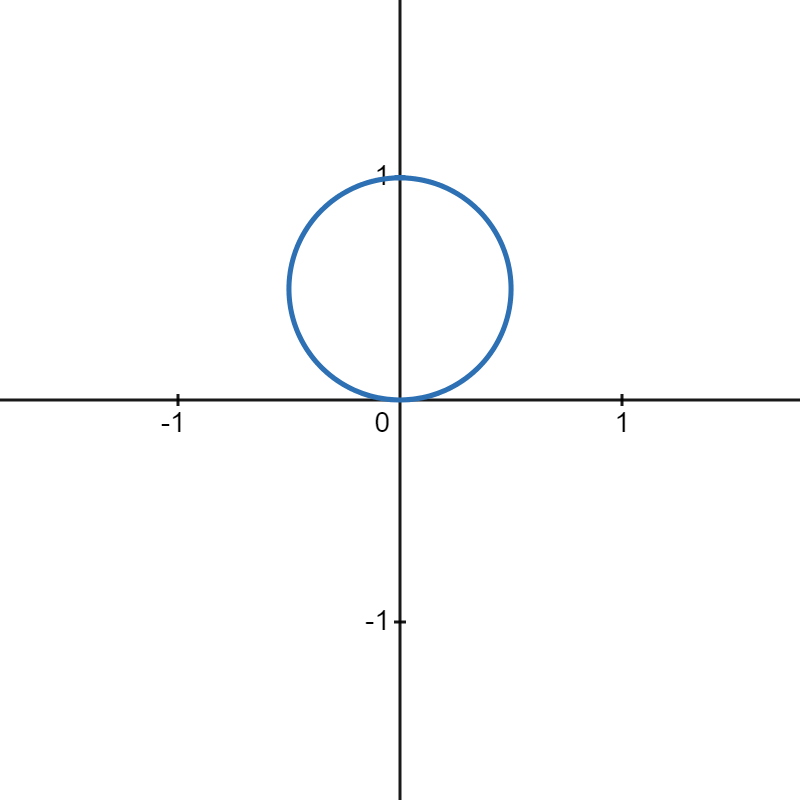}
\end{subfigure}
\hfill 
\begin{subfigure}[b]{0.3\textwidth}
    \centering
    \includegraphics[width=\textwidth]{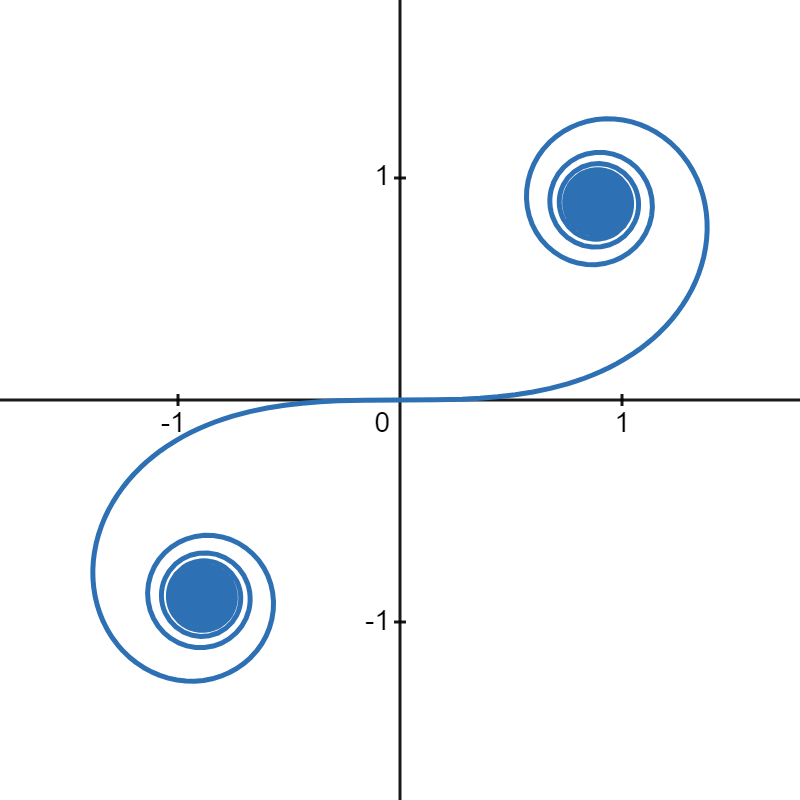}
\end{subfigure}
\caption{
    (a) A straight line with gradient equal to one.
    (b) A circle with curvature equal to two.
    (c) An Eulerian clothoid with curvature equal to arc-length.
}
\label{fig1}
\end{figure}

The right hand side of \eqref{CDF} is the surface Laplacian of the curvature,
and as such naturally generalises to the case of non-graphical curves.
In this setting, more solitons $\gamma:\R\rightarrow\R^2$ are known to exist:
namely circles, Euler's clothoids, the lemniscate of Bernoulli
\cite{EGMWW16} (see figures \ref{fig1} and \ref{fig2}), and very recently, the grim raindrop \cite{OW24}.
The lemniscate shrinks, corresponding to a non-graphical backward self-similar solution, and the grim raindrop translates, corresponding to a non-graphical travelling wave.

Let us briefly explain the key ingredients of the proof of Theorem \ref{TMmain}.
From a big-picture point of view, Theorem \ref{TMmain} follows from the fact
that the curvature of a graphical soliton must be either identically zero, or bounded strictly away from zero on large intervals.
This is impossible (see Lemma \ref{LMtoomuchangle}) and prevents the solution from existing.
Thus the curvature must vanish, which implies that the solution is of the form $x\mapsto Ax$.
The main difficulty thus becomes how to show that the curvature does indeed stay away from zero on an interval of large enough size.
We discovered that certain associated functions (for self-similar profiles, they are $Q$, see \eqref{df-Q} 
and for travelling waves they are $M$, see \eqref{EQdefnofM}) are convex, in a sense.
For travelling waves, we also need to apply some symmetry reductions to keep the
number of different cases tractable.

\section*{Acknowledgements}

The majority of this work was completed while the second author was visiting the University of Warsaw, he is grateful for their kind hospitality.
Both authors enjoyed a partial support from the National Science Centre, Poland, through the grant 2017/26/M/ST1/00700. A visit of GW to the University of Warsaw was supported by a microgrant from the IDUB program.

The authors would like to thank Yoshikazu Giga and Sho Katayama for pointing out an error in the original proof of Theorem \ref{thm13}.
The authors also thank the referees whose careful reading helped to improve the paper.

\section{Steady states}
\label{SNss}

In this section we prove the following result.

\begin{lem}
\label{LMsteadystateclassification}
Suppose that $\phi:\R\rightarrow\R$ is a steady state profile.
Then $\phi(x) = Ax$ for some $A\in\R$.
\end{lem}

First, let us make precise our notion of steady state profile.
For the definition we need the notation $\LocSmooth$ for functions that
are infinitely continuously differentiable at all $x\in\R$.
We also use the following convenient shorthand.

\begin{defn}
Define $\SL:\LocSmooth \to \LocSmooth$ by
\[
\SL[\phi] = 
 -\frac{d}{d x} \bigg(
	\frac{1}{\sqrt{1+(d \phi/d x)^2}}
	\frac{d}{d x}\bigg(
		\frac{d^2 \phi/d x^2}{\sqrt{1+(d \phi/d x)^2}^3}
	\bigg)
\bigg)\,.
\]
\end{defn}

\begin{defn}
We call $\phi\in\LocSmooth$ a \emph{steady state} profile if and only if
\[
	\SL[\phi] = 0\,.
\]
\end{defn}

It will be helpful to introduce some suggestive shorthand.
Set
\[
k[\phi] = 
		\frac{d^2 \phi/d x^2}{\sqrt{1+(d \phi/d x)^2}^3}
\,,
\]
and $v[\phi] = \sqrt{1 + (d \phi/d x)^2}$.
Note that $v[\phi](x) \ge 1$ for all $x$.
The essential idea behind this non-existence result, and all of the rest in this paper, is the following lemma.

\begin{lem}
\label{LMtoomuchangle}
Let $u:\R\rightarrow\R$ be a graph.
Then
\[
\sup_{a,b\in\R}\bigg|\int_a^b k[u]\,v[u]\,dx\bigg| \le \pi\,.
\]
\end{lem}
\begin{proof}
We calculate
\[
\int_a^b k[u]\,v[u]\,dx
	= \int_a^b \frac{d}{d x}\arctan\bigg(\frac{d u}{d x}\bigg)\,dx
	= \arctan B - \arctan A
\]
where $A = \frac{d u}{d x}(a)$ and $B = \frac{d u}{d x}(b)$.
Therefore
\[
\sup_{a,b\in\R}\int_a^b k[u]\,v[u]\,dx \le \sup_{A,B\in\R}
	(\arctan B - \arctan A) = \pi\,,
\]
as required.
\end{proof}

\begin{figure}[t]
\centering
\includegraphics[width=0.65\textwidth,trim=0cm 14cm 0cm 14cm,clip]{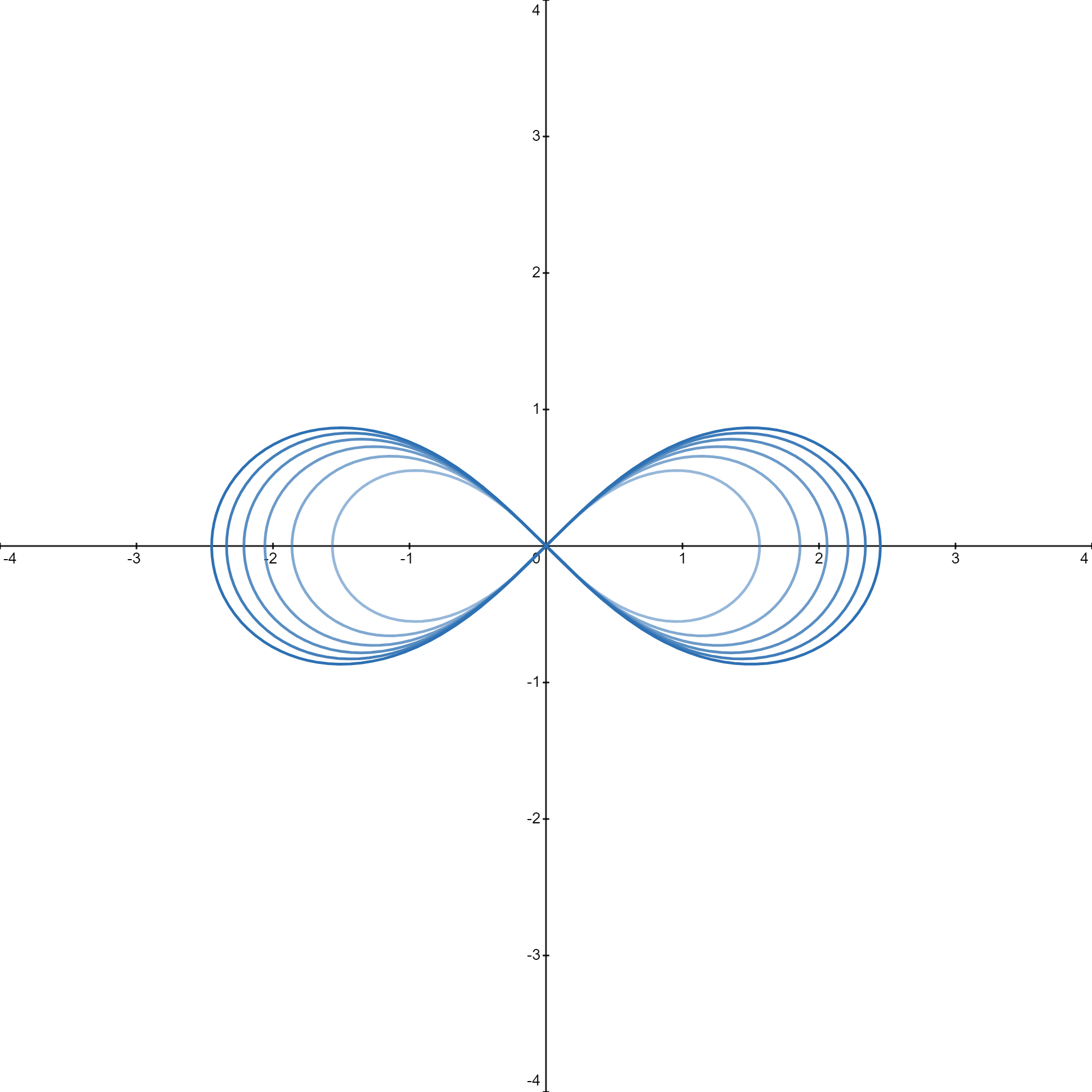}
\caption{The lemniscate of Bernoulli shrinking under surface diffusion flow.
The figure overlays images of the flow at times $0, 1, 2, 3, 4, 5$, and $6$, by
which time it has vanished to the origin.
Its initial parametrisation is
\begin{minipage}[t]{\linewidth}
    \[
    t\mapsto
\frac{\sqrt{6}}{1+\sin^{2}t}\left(\cos\ t,\ \frac{1}{2}\sin\ 2t\right)
.
    \]
\end{minipage}
}
\label{fig2}
\end{figure}

\begin{proof}[Proof of Lemma \ref{LMsteadystateclassification}]
The steady state profile hypothesis is equivalent to
\[
\frac{d}{d x} \bigg( \frac1{v[\phi]} \frac{d k[\phi]}{d x} \bigg)
 = 0
\,,
\]
where $k[\phi]$ and $v[\phi]$ were introduced above.
Therefore,
\[
\frac1{v[\phi]} \frac{d k[\phi]}{d x} = a
\]
where $a\in\R$ is a constant.
Furthermore
\[
k[\phi](x) = a\int_0^x v[\phi]\,dy + b
\]
where $b\in\R$ is another constant.

Now we separate out three cases.

{\bf Case 1: $a=0$.} Then
\[
k[\phi]\,v[\phi]
	= av[\phi]\int_0^x v[\phi]\,dy + bv[\phi]
	= bv[\phi]
	\begin{cases}
		\ge b\,,\text{ for $b>0$},\\
		= 0\,,\text{ for $b=0$},\\
		\le b\,,\text{ for $b<0$}\,.
	\end{cases}
\]
If $b\ne0$, we thus have
\[
\bigg|\int_0^{2\pi/b} k[\phi]\,v[\phi]\,dx\bigg|
	\ge b(2\pi/b) > \pi\,.
\]
This is a contradiction with Lemma \ref{LMtoomuchangle}.
If $b=0$, then we have
\begin{equation}
\label{EQkconstss}
k[\phi]\,v[\phi]
 = \frac{d}{dx}\arctan\bigg(\frac{du}{dx}\bigg)
 = 0
\end{equation}
so
\begin{equation}
\label{EQkconstss2}
 \frac{d u}{d x}
	= A,
\end{equation}
where $A\in\R$.
These are the only solutions allowed for the result.

{\bf Case 2: $a>0$.} Then
\[
k[\phi]\,v[\phi]
	= av[\phi]\int_0^x v[\phi]\,dy + bv[\phi]
	\ge av[\phi](x - |b|)
	\ge a
\]
for all $x>|b|+1$.
Therefore
\begin{equation}
\label{EQtoomuchangleapos}
\bigg|\int_{|b|+1}^{|b| + 1 + 2\pi/a} k[\phi]\,v[\phi]\,dx\bigg|
	\ge a(2\pi/a) > \pi\,,
\end{equation}
and we have a contradiction with Lemma \ref{LMtoomuchangle}.

{\bf Case 3: $a<0$.} Then, similarly to the above,
\[
k[\phi]\,v[\phi]
	= av[\phi]\int_0^x v[\phi]\,dy + bv[\phi]
	\le -|a|v[\phi](x - |b|)
	\le -|a|
\]
for all $x>|b|+1$.
Integrating the above on the interval $I = (|b|+1+2\pi/a, |b|+1)$ then gives a contradiction as before.

This finishes the proof.
\end{proof}

\section{Self-similar solutions}

Let us be precise about what we mean by self-similar profile.
We start with the notion of self-similar solution.

\begin{defn}
Suppose $u\in C^\infty_{loc}(\R\times(0,\infty))$ satisfies \eqref{CDF}.
Set for any $x\in\R$, $t\ne0$ and $\lambda>0$
\[
	u^\lambda(x,t) = \lambda^{-1}u(\lambda x, \lambda^4t)\,.
\]
We call $u$ a \emph{forward self-similar solution} if and only if $u^\lambda(x,t) = u(x,t)$ for all $x$ and $t>0$.
\end{defn}

\begin{defn}
Suppose $u\in C^\infty_{loc}(\R{\times(-\infty,0)})$ satisfies \eqref{CDF}.
We call $u$ a \emph{{backward} self-similar solution} if and only if $u^\lambda(x,t) = u(x,t)$ for all $x$ and $t<0$.
\end{defn}

The notions of forward and backward self-similar profile are used in e.g. \cite{GGS10}.
In geometric flows these are often termed expanders and shrinkers (and are fundamental examples of immortal and ancient solutions respectively).

\begin{rmk}
The rescaling $u\mapsto u^\lambda$ preserves the solution property of $u$.
\end{rmk}

Now we define the self-similar profile.

\begin{defn}
Suppose $u\in C^\infty_{loc}(\R\times(0,\infty))$ is a forward self-similar solution.
Set $\phi:\R\to\R$ by requiring for all $t>0$, $x\in \R$,
\begin{equation}
\label{EQsspdefn}
	\phi(t^{-\frac14}x) = t^{-\frac14}u^{t^{-\frac14}}(x,t)
	= u(t^{-\frac14}x,1)
\,.
\end{equation}
If $u\in C^\infty_{loc}(\R\times(-\infty,0))$ is a backward self-similar solution, we define $\phi$ by replacing $t$ with $-t$ in \eqref{EQsspdefn}.
The function $\phi$ is called the \emph{forward/backward self-similar profile} of $u$.
\end{defn}

We typically use $y = |t|^{-\frac14}x$ for the independent variable of $\phi$.

\begin{rmk}
Clearly a self-similar profile is determined by a self-similar solution.
Similarly, a self-similar solution is determined by a self-similar profile.
Given a self-similar forward profile $\phi$, we set
\[
	u(x,t) = t^{\frac14}\phi(t^{-\frac14}x)\,,
\]
and if $\phi$ is a backward profile, set $u(x,-t) = t^{\frac14}\phi(t^{-\frac14}x)$, for all $x\in\R$, $t>0$.
In this way, the map $u \leftrightarrow \phi$ is a bijection.
\end{rmk}

The self-similar profile satisfies a particular differential equation, which we now derive.

\begin{lem}
\label{LMssp}
A forward self-similar profile $\phi:\R\to\R$ satisfies
\begin{equation}
\label{EQsspgeom}
\frac{d}{dy}\bigg(
    \frac1{v[\phi]}
    \frac{dk}{dy}
    \bigg)
     = -\frac14\,
     \bigg(
     \phi(y)
     - y\frac{d\phi}{dy}
     \bigg)
     \,.
\end{equation}
A backward self-similar profile satisfies \eqref{EQsspgeom} with the opposite sign on the right-hand side.
\end{lem}
\begin{proof}
First, let us consider the case of a forward self-similar profile. Using the chain rule
\[
\frac{\partial}{\partial t}\Big( t^{\frac14}\phi(y) \Big)
	= \frac14 t^{-\frac34} \Big(\phi - y\frac{d\phi}{dy}\Big)
\,.
\]
Second, using the self-similarity property $u = u^\lambda$ and the definition of $\phi$,
\[
\frac{\partial}{\partial t}\Big( t^{\frac14}\phi(y) \Big)
	= -\frac{\partial}{\partial x}\bigg(
		\frac1{v[u^\lambda]}
		\frac{\partial k[u^\lambda]}{\partial x}
	\bigg)
	= -\lambda^3\frac{d}{dy}\bigg(
		\frac1{v[\phi]}
		\frac{dk[\phi]}{dy}
	\bigg)
\,.
\]
Equating both sides (recall $\lambda = t^{-\frac14}$) yields
\begin{equation}
\label{EQsssde}
\frac{d}{dy}\bigg(
	\frac1{v[\phi]}
	\frac{dk[\phi]}{dy}
\bigg)
= -\frac14 \Big(\phi(y) - y\frac{d\phi}{dy}\Big)
\,,
\end{equation}
as required.
The case of a backward self-similar profile is analogous and left to the reader.
\end{proof}

Set now, for $B\in\{-1,1\}$,
\begin{equation}\label{df-Q}
Q^B(y)
= k^2(y)
+ B\frac14\Bigg[
(\phi^2(y) + y^2)
- \phi^2(0) - 2|\phi(0)|\int_0^yv[\phi](\hat y)\,d\hat y
- \bigg(
    \int_0^y v[\phi](\hat y)\, d\hat y
\bigg)^2
\Bigg]
\,.
\end{equation}
The next lemma shows that the functions
$Q^{\pm1}:\R\to\R$ are convex along forward/backward self-similar profiles.
\begin{lem}
\label{LMqconvex}
Let $\phi:\R\to\R$ be a forward/backward self-similar profile.
Then $Q^{\pm1}$ satisfies
\[
\frac1{v[\phi]}\frac{d}{dy}\bigg(
    \frac1{v[\phi]}
    \frac{dQ^{\pm1}}{dy}
    \bigg)
= 2\bigg(
    \frac1{v[\phi]}
    \frac{dk}{dy}
\bigg)^2
\ge 0\,.
\]
\end{lem}
\begin{proof}
First, note that
\[
\frac1{v[\phi]}\frac{d}{dy}\bigg(
    \frac1{v[\phi]}
    \frac{d}{dy}
    \bigg(
- \phi^2(0) - 2|\phi(0)|\int_0^yv[\phi](\hat y)\,d\hat y
    \bigg)
    \bigg)
= 0\,.
\]
So, we may ignore these terms in our calculation. 
Differentiating, 
\allowdisplaybreaks
\begin{align*}
\frac1{v[\phi]}\frac{d}{dy}\bigg(
    \frac1{v[\phi]}
    \frac{dQ^B}{dy}
    \bigg)
&=     
\frac1{v[\phi]}\frac{d}{dy}\bigg(
    \frac1{v[\phi]}
    \frac{d}{dy}
\bigg(
k^2(y)
+ \frac{B}4(\phi^2(y) + y^2)
- \frac{B}4\bigg(
    \int_0^y v[\phi](\hat y)\, d\hat y
\bigg)^2
\bigg) 
\bigg) 
\\
&=
\frac1{v[\phi]}\frac{d}{dy}\bigg(
    \frac1{v[\phi]}
    \frac{d}{dy}
\bigg(
k^2(y)
\bigg) 
\bigg) 
\\&\qquad
+
B\frac1{v[\phi]}\frac{d}{dy}\bigg(
    \frac1{v[\phi]}
    \frac{d}{dy}
\bigg(
 \frac14(\phi^2(y) + y^2)
- \frac14\bigg(
    \int_0^y v[\phi](\hat y)\, d\hat y
\bigg)^2
\bigg) 
\bigg) 
\\
&=
2\bigg(
    \frac1{v[\phi]}\frac{dk}{dy}
\bigg)^2
+ 2k
    \frac1{v[\phi]}
    \frac{d}{dy}
\bigg(
    \frac1{v[\phi]}
    \frac{dk}{dy}
\bigg)
\\&\qquad
+
B\frac1{v[\phi]}\frac{d}{dy}\bigg(
    \frac1{v[\phi]}
\bigg(
 \frac12\bigg(
    \phi(y)\frac{d\phi}{dy} 
    + y
    \bigg)
- \frac12v[\phi]
    \int_0^y v[\phi](\hat y)\, d\hat y
\bigg) 
\bigg) 
\intertext{
If $\phi$ is a forward self-similar profile, set $B=1$, and if $\phi$ is a backward self-similar profile, set $B=-1$.
Then, using Lemma \ref{LMssp}, and  $\frac{d}{dy}\frac1{v[\phi]}
= -k\frac{d\phi}{dy}$, we calculate:}
\frac1{v[\phi]}\frac{d}{dy}\bigg(
    \frac1{v[\phi]}
    \frac{dQ^B}{dy}
    \bigg)
&=
2\bigg(
\frac1{v[\phi]}\frac{dk}{dy}
\bigg)^2
+ 2k
    \frac1{v[\phi]}
\bigg(
-B\frac14
     \bigg(
     \phi(y)
     - y\frac{d\phi}{dy}
     \bigg)
\bigg)
\\&\qquad
+
B\frac12
\frac1{v[\phi]}\frac{d}{dy}\bigg(
    \frac1{v[\phi]}
\bigg(
    \phi(y)\frac{d\phi}{dy} 
    + y
    \bigg)
\bigg) 
- B\frac12\frac1{v[\phi]}\frac{d}{dy}\bigg(
    \int_0^y v[\phi](\hat y)\, d\hat y
\bigg) 
\\
&=
2\bigg(
\frac1{v[\phi]}\frac{dk}{dy}
\bigg)^2
+ B\Bigg[
- \frac12
    k
    \frac1{v[\phi]}
\bigg(
     \phi(y)
     - y\frac{d\phi}{dy}
\bigg)
\\&\qquad
+
\frac12
\frac1{v[\phi]}\bigg(
    -k\frac{d\phi}{dy}
\bigg(
    \phi(y)\frac{d\phi}{dy} 
    + y
    \bigg)
\bigg) 
\\&\qquad
+
\frac12
\frac1{v[\phi]}\bigg(
    \frac1{v[\phi]}
\bigg(
    \phi(y)\frac{d^2\phi}{dy^2} 
    + \bigg(\frac{d\phi}{dy}\bigg)^2 
    + 1
    \bigg)
\bigg) 
- \frac12
\Bigg]
\\
&=
2\bigg(
\frac1{v[\phi]}\frac{dk}{dy}
\bigg)^2
+ B\Bigg[
- \frac12
    k
    \frac1{v[\phi]}
\bigg(
     \phi(y)
     - y\frac{d\phi}{dy}
\bigg)
\\&\qquad
- \frac12
k
\frac1{v[\phi]}
\frac{d\phi}{dy}
\bigg(
    \phi(y)\frac{d\phi}{dy} 
    + y
    \bigg)
\bigg) 
+
\frac12
v[\phi]
k
\phi(y)
\Bigg]
\\
&=
2\bigg(
\frac1{v[\phi]}\frac{dk}{dy}
\bigg)^2
+ B\frac12v[\phi]k\phi(y)
\bigg(
1
- \frac{1 - y\frac{d\phi}{dy} + (\frac{d\phi}{dy})^2 + y\frac{d\phi}{dy}}{v^2[\phi]}
\bigg)
\\
&=
2\bigg(
\frac1{v[\phi]}\frac{dk}{dy}
\bigg)^2
\,,
\end{align*}
as required.
\end{proof}

Our main result is the following rigidity statement.

\begin{thm}\label{thm13}
(1) All forward self-similar profiles \GW{satisfying $\phi(0) = 0$} are trivial, that is, $\phi(y) = Ay$ for some $A\in\R$.

(2) All \GW{forward and} backward self-similar profiles satisfying
\begin{equation}
\label{EQalmoststraight}
(\phi^2(y) + y^2)
- \phi^2(0) - 2|\phi(0)|\int_0^y v[\phi](\hat y)\,d\hat y
- \bigg(
    \int_0^y v[\phi](\hat y)\, d\hat y
\bigg)^2\text{ bounded for } y\in\R\,,
\end{equation}
are also trivial.
\label{TMrigiditysss}
\end{thm}
\begin{proof}
(1) First, note the estimate
\begin{equation}
\phi^2(y) + y^2
\le |\phi(0)|^2 + 2|\phi(0)|\GW{\bigg|}\int_0^y v[\phi](\hat y)\,d\hat y\GW{\bigg|}
+ \bigg(\int_0^y v[\phi](\hat y)\,d\hat y\bigg)^2
 \,,\quad\text{ for all $y\in\R$}\,.
\label{EQestimate}
\end{equation}
We present 
a vector-based derivation here.
Start with the equality
\[
\begin{pmatrix} y \\ \phi(y) \end{pmatrix}
= 
\begin{pmatrix} 0 \\ \phi(0) \end{pmatrix}
+ \int_0^y \begin{pmatrix} 1 \\ \frac{d\phi}{dy} \end{pmatrix}
\,dy
\]
and then estimate:
\[
{\phi^2(y) + y^2}
\le \bigg(|\phi(0)| + \PR{\bigg|}\int_0^y v[\phi](\hat y)\,d\hat y\PR{\bigg|}\bigg)^2
= |\phi(0)|^2 + 2|\phi(0)|\GW{\bigg|}\int_0^y v[\phi](\hat y)\,d\hat y\GW{\bigg|}
+ \bigg(\int_0^y v[\phi](\hat y)\,d\hat y\bigg)^2
\,.
\]
The estimate \eqref{EQestimate}, \PR{combined with $B=1$ and $\phi(0) = 0$,}  implies 
\[
Q(y) \le k^2(y)\,.
\]
\GW{This is where we use the assumption that we are in the forward case.}
Since $Q$ is convex, we know that $Q(y)\rightarrow\infty$ as either $y\to\infty$ or $y\to-\infty$, unless $Q$ is constant.
As $k^2(y) \ge Q(y)$, we thus have that $k^2(y)\to\infty$ as $y\to\infty$ or $y\to-\infty$.
This means that $k(y)\to\pm\infty$ as $y\to\infty$ or $y\to-\infty$.
This is a contradiction with Lemma \ref{LMtoomuchangle}.
If $Q$ is constant, then $\frac{dk}{dy} =  0$ and our claim follows again.

(2) This is similar to the conclusion of case (1) above: Since $k^2(y) = Q(y)$ up to a bounded function (by hypothesis \eqref{EQalmoststraight}), we thus have that $k^2(y)\to\infty$ as $y\to\infty$ or $y\to-\infty$.
This is a contradiction with Lemma \ref{LMtoomuchangle}.
\end{proof}

\begin{rmk}
The hypothesis \eqref{EQalmoststraight} implies, intuitively, that the positive and negative rays of the profile differ from being straight by a bounded function.
For a concrete example, consider the (not smooth)  function $\psi_{A,B}:\R\to\R$
defined by
\[
\psi_{A,B}(y) = \begin{cases}(y,A|y|)\,,\quad y\ge 0\\
(y,B|y|)\,,\quad y<0\,,
\end{cases}
\]
where $A,B\in\R$.
Along this curve we have 
\[
\psi_{A,B}^2(y) + y^2 = \begin{cases}
y^2(1+A^2)\,,\quad y\ge0\\
y^2(1+B^2)\,,\quad y<0
\end{cases}
\]
and
\[
\bigg(\int_0^y \sqrt{1+\Big(\frac{d\phi_{A,B}}{dy}\Big)^2} \,d\hat y\bigg)^2 = 
\begin{cases}
y^2(1+A^2)\,,\quad y>0\\
y^2(1+B^2)\,,\quad y<0
\end{cases}
\] 
 so \eqref{EQalmoststraight} is (optimally) satisfied.

This shows that a small perturbation of $\psi_{A,B}$ (enabling smooth candidates, as required by our result) also satisfies \eqref{EQalmoststraight}. 
The implication is, by Theorem \ref{TMrigiditysss},  that there is no self-similar profile in a neighbourhood of any scaled absolute value function.
\label{RMK1}
\end{rmk}

The above remark may be compared with \cite[Theorem 3.12]{KL12}.
Note that the statement there is explicitly written for the Willmore flow, but they remark earlier that the same result holds for surface diffusion.
They claim that there is a unique non-trivial self-similar profile in a neighbourhood of a self-similar function.
In one space dimension, the self-similar functions are those given in Remark \ref{RMK1} above.
So, our Theorem \ref{TMrigiditysss} shows that the only self-similar profiles close to the self-similar functions in one space dimension is $\gamma(y) = (y,0)$: the trivial solution.

The solutions proved to exist in \cite{KL12} are in homogeneous Lipschitz spaces, which means that while the derivative is well-defined, the value of the solution is at each moment in time is determined only up to a constant.
This makes prescribing data (such as self-similar data) for the solution problematic.
Instead, one may prescribe the value of the derivative of the solution, and integrate the equation.
This introduces a constant dependent on time.
So, while there do not exist non-trivial self-similar profiles in one dimension that are close to self-similar functions (by Theorem \ref{TMrigiditysss} and Remark \ref{RMK1}), an approach similar to that of \cite[Theorem 3.12]{KL12} could yield the existence of special solutions to surface diffusion flow that are, up to the addition of a non-zero time-dependent function, self-similar.
We note that while the surface diffusion flow is invariant under vertical translation (the addition of a constant to the graph function), the self-similar profile equation \eqref{EQsspgeom} is not.

\begin{rmk}
In this article we have used the graphical parametrisation to express all of our results.
In this remark, we give some equivalent expressions using geometric language, and also state some additional identities that we feel may be useful in extending Theorem \ref{TMrigiditysss}.
We only consider the forward case here, leaving the backward case to the reader.
Let $\gamma(y) = (y,\phi(y))$ be the graph of $\phi$.
We may parametrise $\gamma$ by arclength $s$, using
\begin{equation}\label{df-s}
s(y) = \int_0^y \sqrt{1+\phi_y^2}\,dy
\,.
\end{equation}
In this remark, we use subscripts to denote differentiation.

Note that $s:\R\to\R$ is a diffeomorphism.
The arclength derivative operator $\partial_s$ is
\[
    \partial_s
    = \frac1{v[\phi]}\partial_y
    \,.
\]
Thus the LHS of \eqref{EQsssde} is $v[\phi]\partial_s\partial_sk = v[\phi]k_{ss}$.
For the RHS, we calculate that the unit tangent along $\gamma$ is
\[
T = \partial_s\gamma
 = \frac1{v[\phi]}\bigg(1,\frac{d\phi}{dy}\bigg)
\]
so
\[
N = \frac1{v[\phi]}\bigg(-\frac{d\phi}{dy},1\bigg)
\]
and thus
\[
\gamma\cdot N
 = \frac1{v[\phi]}\bigg(
    - y\frac{d\phi}{dy}
    + \phi
 \bigg)
 \,.
\]
Therefore the RHS of \eqref{EQsssde} is $-\frac14v[\phi]\gamma\cdot N$.
Then the statement of Lemma \ref{LMssp} 
reads
 \[
k_{ss} = -\frac14\, \gamma\cdot N\,.
 \]
Hence, after taking one more derivative
$$
k_{sss} = \frac k4 \gamma\cdot T.
$$
Moreover,
 \[
 Q_{ss} = 2k_s^2\,,
 \quad
 \text{where}
 \quad
 Q(s) = k^2(s) - \frac14(|\gamma|^2(s) - s^2)
 \,.
 \]
Along a self-similar profile, the tangent and normal vectors satisfy a constraint and the curvature scalar a particular fourth-order nonlinear ODE.
In particular, we have the following facts:
(1) $k_{sss}/k$ is a smooth function, (2) the curvature satisfies
\begin{equation}
\label{EQkfourth}
k_{s^4} = k_{s}\frac{k_{sss}}{k} - k_{ss}k^2 + \frac14k\,,
\end{equation}
and (3) the tangent and normal components of $\gamma$ (equivalently the second derivative $k_{ss}$ and the function $k_{sss}/k$) satisfy the following constraint
\[
\frac{|\gamma|^2}{16}
=
    \frac1{16}(\gamma\cdot T)^2
    + \frac1{16}(\gamma\cdot N)^2
=
k_{ss}^2 + \frac{k_{sss}^2}{k^2}
\,.
\]
\end{rmk}

\section{Travelling waves}

\begin{defn}
A solution $u:\R\times\R\to\R$ to \eqref{CDF} is called a \emph{travelling wave} with direction $e=(a,b)$ if and only if
$u(x,t) = \phi(x-at)+bt$.
The function $\phi:\R\to\R$ is called the \emph{travelling wave profile}.
\end{defn}

\begin{thm}
\label{PNnonontrivialwaves}
Let $\phi$ be the travelling wave profile of a travelling wave solution $u$ with $e=(a,b)$ and $|e| = \sqrt{a^2+b^2} \ne 0$.
Then $\phi(x) = bx/a$, where $a\neq 0$.
\end{thm}

\begin{proof}
First, we derive a differential equation for the profile.
Observe that
\[
\frac{\partial}{\partial t} (\phi(x-at)+bt) = -a\frac{d\phi}{dy} + b = -\frac{d}{dy}\bigg(\frac1{v[\phi]}\frac{dk[\phi]}{dy}\bigg)
\]
or
\begin{equation}
\label{EQeqfortranprof}
 \frac{d}{dy}\bigg(\frac1{v[\phi]}\frac{dk[\phi]}{dy}\bigg)
 =
 a\frac{d\phi}{dy} - b\,.
\end{equation}
{\it Step 1.}
We now rule out $a=0$.
First, if $a=0$ then $b\ne0$, as we assumed $|e|\ne0$.
If $a=0$, then integration of \eqref{EQeqfortranprof} over $(y_1, y)$ yields
\begin{align*}
\frac{dk[\phi]}{dy}(y) &=  - b(y-y_1) v[\phi](y) + \frac{dk[\phi]}{dy}(y_1) \frac{v[\phi](y)}{v[\phi](y_1)} \\ &=
(-by + c_1)v[\phi], 
\end{align*}
where
$$
c_1 =  \frac 1{v[\phi](y_1)}\frac{dk[\phi]}{dy}(y_1) + b y_1
\in\R
$$ 
and  $y_1\in \R$ is any number.
This implies {that for any $\alpha>0$ we can bound $\frac{dk[\phi]}{dy}$ from below by $\alpha$} on an infinite interval. 
Indeed,
$$
\frac{dk[\phi]}{dy}>\alpha\quad\hbox{for}\quad 
\begin{array}{ll}
    y<(c_1-\alpha)/b & \hbox{if }b>0, \\
    y> (c_1-\alpha)/b & \hbox{if }b<0. 
\end{array}
$$
If $\frac{dk[\phi]}{dy} \ge {\alpha}$ on an infinite interval then $k \ge 1$ on another infinite
interval and we have a contradiction via Lemma \ref{LMtoomuchangle}, as before.

{\it Step 2.}
Now we assume $a\ne0$ and
we set
\begin{equation}
\label{EQdefnofM}
M(y) = k^2[\phi](y) + 2(ay + b\phi(y))
\,.
\end{equation}
We calculate
\begin{equation}\label{r-dM}
\frac1{v[\phi]} \frac{dM}{dy}
	= 2k[\phi] \frac1{v[\phi]} \frac{dk[\phi]}{dy}
	 + \frac{2}{v[\phi]}\Big(a + b\frac{d\phi}{dy}\Big)
\end{equation}
and
\begin{align}
	\frac1{v[\phi]}\frac{d}{dy}\bigg(\frac1{v[\phi]} \frac{dM}{dy}\bigg)
	&= 2\bigg(\frac1{v[\phi]} \frac{dk[\phi]}{dy}\bigg)^2
	 + 2\frac{k[\phi]}{v[\phi]} \frac{d}{dy}\bigg(\frac1{v[\phi]} \frac{dk[\phi]}{dy}\bigg)
	 + 2\frac{1}{v[\phi]}\frac{d}{dy}\bigg(\frac{1}{v[\phi]}\Big(a + b\frac{d\phi}{dy}\Big)\bigg)
\notag\\
	&= 2\bigg(\frac1{v[\phi]} \frac{dk[\phi]}{dy}\bigg)^2
	 + 2\frac{k[\phi]}{v[\phi]}\bigg(
		a\frac{d\phi}{dy} - b
		- a\frac{d\phi}{dy} - b\bigg(\frac{d\phi}{dy}\bigg)^2
		+ bv^2[\phi]
	\bigg)
\notag\\
	&= 2\bigg(\frac1{v[\phi]} \frac{dk[\phi]}{dy}\bigg)^2
\,. 
\label{EQconvexityM}
\end{align}
Integrating \eqref{EQconvexityM} twice, 
first over $(y_1,x)$, then over $(y_2, y)$, where $y_2> y_1$ are arbitrary yields
\begin{equation}
\label{EQestforM}
	M(y) \ge c\int_{y_2}^yv[\phi]\,dx + d
\end{equation}
where $c,d\in\R$ are defined below and $y>y_2$ 
\[
	c = c(y_1) =\frac{1}{v[\phi](y_1)}\frac{dM}{dy}(y_1)\,,\quad
	d = d(y_2) = M(y_2)
\]
with $y > y_2> y_1$. 

Note that \eqref{EQconvexityM} implies that for any $y_3< y_4$ we have
\begin{equation}
\label{EQintksbdd}
2\int_{y_3}^{y_4}
		\frac1{v[\phi]}\bigg(\frac{dk[\phi]}{dy}\bigg)^2 
		\,dy
	=
	\frac1{v[\phi]} \frac{dM}{dy}(y_4)
	- \frac1{v[\phi]} \frac{dM}{dy}(y_3).
\end{equation}
{\it Step 3.} Now, we claim that $\frac1{v[\phi]} \frac{dM}{dy}$ cannot be constant on any open interval, {unless $\phi$ is a linear function}. Indeed,
if there exist $y_3$, $y_4$ with $y_3<y_4$ such that $c(y_3) = c(y_4)$ then, we infer from \eqref{EQintksbdd} 
that
$\frac{dk[\phi]}{dy}$ vanishes on $(y_3,y_4)$. Since $\phi$ is an analytic solution {to} (\ref{EQeqfortranprof}), then $k$ is also analytic. Hence, vanishing of $\frac{dk[\phi]}{dy}$ on $(y_3,y_4)$ 
implies 
that
$\frac{dk[\phi]}{dy}$ vanishes on all of $\R$.
Then $k[\phi] = c_1$ and arguing as before we deduce
that $\phi(y) = by/a$, as required.

{\it Step 4.} 
Due to Step 3, we may assume that $M(y)$ is not only convex, but its gradient {with respect to the arclength parameter, i.e. $\frac 1v \frac{dM}{dy}(y),$} is strictly increasing, because $k$ may not be constant on any interval, {unless $k=0$}. 

{
Now, we claim that $\frac1{v[\phi]} \frac{dM}{dy}$ must be bounded. Let us suppose otherwise. In this case
\eqref{EQestforM} implies, 
\begin{equation}
\label{EQestfork2}
k^2(y) \ge k^2(y_2) + 2(a(y_2-y)+b(\phi(y_2)-\phi(y))) + c(y_1)\int_{y_2}^yv[\phi]\,dx
\,.
\end{equation}
Estimating like so
\[
\big|
a(y_2-y)+b(\phi(y_2)-\phi(y))
\big|
 = \bigg|\int_{y_2}^y \frac{a+b\frac{d\phi}{dy}}{v[\phi]}\,v[\phi]\,dy\bigg|
 \le \sqrt{a^2+b^2}\int_{y_2}^yv[\phi]\,dx
\]
we find
\[
k^2(y) \ge k^2(y_2) + (c(y_1) - 2\sqrt{a^2+b^2})\int_{y_2}^yv[\phi]\,dx
       \ge k^2(y_2) + (c(y_1) - 2\sqrt{a^2+b^2})(y-y_2)
\,.
\]
Clearly then, as $c(y_1)$ is unbounded, the above yields $k^2(y)$ unbounded, which is a contradiction via Lemma \ref{LMtoomuchangle}.}
{\it Step 5.} Now, from \eqref{EQeqfortranprof} we have
\begin{equation}\label{rE}
 \bigg|
	\frac1{v[\phi]}\frac{d}{dy}\bigg(\frac1{v[\phi]}\frac{dk[\phi]}{dy}\bigg)
 \bigg|
 =
 \bigg|
 	\frac{a\frac{d\phi}{dy} - b}{v[\phi]}
 \bigg|
	\le \sqrt{a^2+b^2}=:E
\,.
\end{equation}
We also notice that due to \eqref{EQintksbdd}, boundedness of $\frac1 {v[\phi]} \frac {dM}{dy}(y)$ is equivalent to $\frac1 {v[\phi]} (\frac {dk}{dy})^2\in L^1(\bR)$. 
If we introduce the arclength parameter $s$ by \eqref{df-s}, then we see that
\begin{equation}
\label{EQsigh}
\infty> \int_\bR \frac1 {v[\phi]} \bigg(\frac {dk}{dy}\bigg)^2\, dy =
\int_\bR \bigg( \frac{\partial k}{\partial s}\bigg)^2 v[\phi]\, dy =
\int_\bR \bigg( \frac{\partial k}{\partial s}\bigg)^2 \, ds
\end{equation}
Now, \eqref{rE} is equivalent to the boundedness of $\frac{\partial^2 k}{\partial s^2}$.
Actually, more is true:

\begin{lem}\label{kL}
Let us assume that $k\in C^2(\bR)$ is such that:\\
$$(a)\quad |k_{ss}|\le E,\qquad (b)\quad \int_\bR k^2_s\,ds <\infty,\qquad(c)\quad\sup_{a,b\in\bR} \left| \int_a^b k\,ds \right|<\pi,
$$
then\\
(i) there is $C_1>0$ such that $|k_s|\le C_1$ and $\displaystyle{\lim_{|s|\to \infty}} k_s(s) =0$; and\\
(ii) there is $C_0>0$ such that $|k|\le C_0$.
\end{lem} 

{\it Step 6.} We proceed assuming validity of Lemma \ref{kL}. In order to apply this Lemma we change variable in the curvature function $k$, namely we use the arc length parameter  $s$ defined in \eqref{df-s}.  
Then, $\frac 1{v[\phi]} \frac{d}{dy} = \frac{d}{ds}$ and \eqref{r-dM} becomes
$$
\frac{dM}{ds} = 2 k[\phi]\frac{dk[\phi]}{ds} + \frac2 {v[\phi]}\left(a + b \frac{d\phi}{dy}\right).
$$
Now, {since $\frac{dM}{ds}$ is monotone and uniformly bounded} we can take the limit as $y\to \infty$ (resp. $y\to - \infty$), then due to Lemma \ref{kL} we obtain
\begin{equation}\label{r-gdM}
M^+ = \lim_{y\to \infty} \frac{dM}{ds} = 
\lim_{y\to \infty} \frac2 {v[\phi]}\left(a + b \frac{d\phi}{dy}\right),\qquad M^- = \lim_{y\to -\infty} \frac{dM}{ds} = 
\lim_{y\to -\infty} \frac2 {v[\phi]}\left(a + b \frac{d\phi}{dy}\right)
\end{equation} 
{\it Step 7.}
We claim that \eqref{r-gdM} implies existence of the limits
$$
\lim_{y\to -\infty} \frac{d\phi}{dy}(y) = d^-, \qquad
\lim_{y\to +\infty} \frac{d\phi}{dy}(y) = d^+,
$$
\PR{which could be infinite, when $M^\pm\in\{ -2 b, 2 b\}.$}

Let us define $g:\bR \to \bR$ by formula 
$$
g(x) = \frac{a + b x}{\sqrt{1+ x^2}}.
$$
Then, \eqref{r-gdM} reads as
$$
M^+ =
2\lim_{y\to \infty} g\left(
\frac{d\phi}{dy}\right),\qquad M^- =2 \lim_{y\to -\infty}
 g\left(
\frac{d\phi}{dy}\right).
$$

Obviously, $g$ is not monotone unless $a=0$, but this case has been already excluded in Step 1. The function $g$ has a single maximum when $a>0$ or a single minimum {when $a<0.$ Moreover, 
$$
\lim_{x\to -\infty} g(x) = - b \neq b = \lim_{x\to \infty} g(x) 
$$
unless $b=0$.}

We may assume that $g$ has a minimum i.e. $a<0$. It is easy to see that if $b\neq 0$ and $q\in [-|b|, |b|)$, then there exists exactly one $p\in\bR$ such that $q= g(p)$. As a result, if $M^+\in \PR{(-2} |b|, \PR2 |b|)$ or $M^+ = \PR2\min_{x\in\bR} g(x) = \PR2 g(y_0)$, then the existence of the limit $M^+$ immediately implies existence of $\displaystyle{\lim_{y\to +\infty} g\big(\frac{d\phi}{dy}\big)}$. 

Otherwise we proceed differently, \PR{assuming first that $M^+\not\in\{-2b, 2b\}$}. We notice that 
there are $d_l< d_r$ such that 
$g(d_l) = M^+ = g(d_r)$. 
Suppose that there exist two sequences $y_n<z_n$  converging to infinity such that
$$
\lim_{n\to \infty} \frac{d\phi}{dy}(y_n) = d_l,\qquad
\lim_{n\to \infty} \frac{d\phi}{dy}(z_n) = d_r.
$$
We notice that for all $n$ there is $\xi_n \in (y_n, z_n)$ such that $g(\frac{d\phi}{dy}(\xi_n)) = \min g.$ Hence, 
$$
[\min g, M^+]\subset  g(\frac{d\phi}{dy})(y_n,z_n),
$$
but this implies that $2g(\frac{d\phi}{dy})$ cannot converge to $M^+$, which contradicts \eqref{r-gdM}. 

\PR{Now, we deal with the excluded values of $M^+$. For the sake of simplicity we assume that $b>0,$ we proceed in a similar way when $b<0.$ If $M^+ =2b= \lim_{\xi\to\infty}2 g(\xi)$, then we immediately see that 
$$
\lim_{y\to \infty} \frac{d\phi}{dy}(y) = \infty =: d_r.
$$
If $\frac 12M^+ = -b = \lim_{\xi\to-\infty} g(\xi) = g(y_b)$, for a unique $y_b\in\bR$, then either 
$$
\lim_{y\to \infty} \frac{d\phi}{dy}(y) = -\infty =: d_r
$$
or $\frac{d\phi}{dy}(y)$ is bounded for $y>0$. In this case the argument from the beginning of this step applies yielding existence of the limit
$$
\lim_{y\to \infty} \frac{d\phi}{dy}(y) = y_b =: d_r.
$$}

The same argument applies, when $y\to -\infty$.

{\it Step 8.}
Suppose \PR{that both $d_r$ and $d_l$  are finite and } one of them 
is  not equal to $b/a$.
Then, \eqref{EQeqfortranprof} implies 
\[
 \frac{d}{dy}\bigg(\frac1{v[\phi]}\frac{dk[\phi]}{dy}\bigg)
 \rightarrow
 a d_p - b \ne 0\qquad p\in\{l, r\}.
\]
Thus there exists an interval $I$ with $|I| = \infty$ and we find a contradiction using Lemma \ref{LMtoomuchangle} as
\[
\bigg|\frac1{v[\phi]}\frac{dk[\phi]}{dy}\bigg|
\ge 1\qquad
\text{for all $y\in I$.}
\]
Therefore both $d_r$ and $d_l$ must be equal to the same value: $b/a$.
Thus, from \eqref{r-gdM}, $M^+ = M^-$, and then by \eqref{EQintksbdd}, $k[\phi]$ must be a constant.
The only constant it can be is zero, which implies that $\phi(y) = by/a$.

\PR{If it happens that $d_r$ or $d_l$ is infinite, then an application of the argument presented above shows that it is impossible as it contradicts  Lemma \ref{LMtoomuchangle}. As a result, both $d_r$ and $d_l$ must be finite and the argument above shows that $d_r= d_l = b/a$, i.e. $\phi(y)= by/a$.}
\end{proof}

In order to complete the argument we are going to present a {\it proof of Lemma \ref{kL}}. 
Our argument is for the ray $(0,\infty)$.
The same reasoning is valid for $(-\infty,0)$.

{\it Step 1.}
Let us suppose that there is $\alpha>0$ such that  $k_s$ has a sign on $R:=(\alpha,+\infty)$.
Then, $k$ is monotone automatically and it must be bounded on $R$ and even
$\lim_{y\to \infty} k(y) =0.$ 
Hence, (ii) holds, and it only remains to prove (i).
If in addition $k_s$ is monotone on $(\beta,\infty)$, then $k_s$ is bounded and due to assumption (b) $\lim_{y\to \infty} k_s(y) =0.$ 
It remains to consider:\\ 
($\alpha$) $k_s$ with a sign (say, $k_s\ge 0$) on $R$ and\\
($\beta$) $k_s$ changing sign infinitely many times on $R$.

{\it Step 2.} We claim that case ($\alpha$) implies first that $k_s$ is bounded. 
If we assume that $k_s\ge 0$, then $k$ must be negative increasing with limit zero at $+\infty$. 

Let us suppose otherwise, i.e. there is a sequence $\xi_n$ converging to infinity such that $k_s(\xi_n)\to\infty$ when $n\to \infty$.
Since $k$ is monotone on $R$ we can consider only such indices $n$ that $ - 1 \le k(\xi_n)\le 0$.

Since $|k_{ss}| \le E$ we see that 
the following inequalities hold,
\begin{align*}
k_s (y)& \ge E(y-\xi_n) + k_s(\xi_n)=:\ell_-(y),\qquad\hbox{for } y\le \xi_n\\
k_s (y)& \ge E(\xi_n-y) + k_s(\xi_n)=:\ell_+(y), \qquad\hbox{for } y>\xi_n.
\end{align*}
Let us set $h_n = \frac 1{2E} k_s(\xi_n),$ then 
we have
$$
k_s(\xi_n - h_n) \ge \ell_-(\xi_n - h_n)
= \frac12 k_s(\xi_n), \qquad
k_s(\xi_n + h_n) \ge \ell_+(\xi_n + h_n)
= \frac12 k_s(\xi_n).
$$
Hence, 
$$
\min\{ k_s(y): y\in [\xi_n- h_n ,\xi_n]\} \ge 
\min\{ \ell(y): y\in [\xi_n- h_n ,\xi_n]\}  = \frac12 k_s(\xi_n).
$$
In the case we are considering right now, the curvature $k$ converges to $0$, when $y$ goes to $+\infty.$ Hence for sufficiently large $n$ by the mean value 
theorem there is $\theta\in (\xi_n - h_n, \xi_n)$ such that
$$
1\ge k(\xi_n ) - k(\xi_n -h_n) = k_s(\theta) h_n \ge \min_{y\in [\xi_n- h_n ,\xi_n]} k_s(y) h_n = \frac 1{2E} k_s^2(\xi_n).
$$
We reach a contradiction, because the right hand side goes to $\infty$.
Thus $k_s$ is bounded.

Now, we shall take care of case ($\beta$), showing that $k_s$ is also bounded in this case.
There are sequences $y_n, z_n$ converging to $\infty$ as $n\to \infty$ and such that $k_s(y_n ) = 0 = k_s(z_n)$, $k_s\ge 0$ on $(y_n, z_n)$, and (supposing $k_s$ is not bounded) there is $\xi_n\in(y_n, z_n)$ such that $k_s(\xi_n)\to\infty$ when $n\to \infty.$ {The case $k_s(\xi_n)\to-\infty$ when $n\to \infty$ can be treated in the same manner. The details are left to an interested reader.}

Let us set $h_n = \frac 1{2E} k_s(\xi_n),$ as above.
The definition of $h_n$ implies that $\xi_n\pm h_n\in (y_n, z_n).$
Let us suppose that there is $x_n\in (y_n, z_n)$ such that $k(x_n) =0$. If $k$ does not change sign on $(y_n, z_n),$ then the argument is easier and it is left to the reader.

Let us estimate $\Delta_n k$, the change of curvature over $(y_n, z_n),$ we have
$$
\Delta_n k =\int_{y_n}^{z_n} k_s (s)\, ds \ge 
\int_{\xi_n - h_n}^{\xi_n+h_n} k_s(s)\, ds \ge \frac12 k_s(\xi_n) 2 h_n = \frac1{2E} k^2_s(\xi_n).
$$
We have two cases:\\ (1) $x_n\le \xi_n$; \\ (2) $x_n> \xi_n$.

In the first case we estimate the following integral of $k$:
$$
\int_{\xi_n}^{\xi_n + h_n} k(s)\, ds \ge \frac1{4} k_s(\xi_n) h_n =
\frac1{4E}k_s^2(\xi_n) >\pi.
$$
for sufficiently large $n$. 
Hence, we reach a contradiction with Lemma \ref{LMtoomuchangle}.
In the second case we estimate the integral of $k$ from above, 
\begin{align*}
\int_{\xi_n- h_n}^{\xi_n} k(s)\, ds&=  \int_{\xi_n- h_n}^{\xi_n} k(s) - k(\xi_n) + k(\xi_n)\, ds\le 
-\int_{\xi_n- h_n}^{\xi_n} \int_s^{\xi_n} k_s(\sigma)\, d\sigma ds
\end{align*}
where  used that assumption $x_n>\xi_n$, which implies that $k(\xi_n)<0$. Hence, we reach
\begin{align*}
\int_{\xi_n- h_n}^{\xi_n} k(s)\, ds
&\le -\int_{\xi_n- h_n}^{\xi_n} \int_s^{\xi_n} \frac 12 k_s(\xi_n)\, d\sigma ds  
= -\int_{\xi_n- h_n}^{\xi_n} \frac 12 k_s(\xi_n)(\xi_n -s )\,ds
\\
&= - \frac 14 k_s(\xi_n ) h_n^2
=  - \frac 1{16E^2} k_s^3(\xi_n) < -\pi 
\end{align*}
for sufficiently large $n$. 
Hence, we reach a contradiction with Lemma \ref{LMtoomuchangle}, 
and so there exists a constant $C_1>0$ such that $|k_s|\le C_1$. 
Combining this with assumption (b) we deduce that $\lim_{s\to\infty} k_s(s)  =0.$ 

{\it Step 3.} It remains to show that $k$ is bounded. For this purpose we adjust the argument used above. Now, there are three sequences $\alpha_n$, $\beta_n,\xi_n$ converging to infinity such that $k(\alpha_n) 
= k(\beta_n){\to 0}$, {when $n\to \infty$ and} $k\ge 0$ on $(\alpha_n, \beta_n)$ and  $\xi_n \in (\alpha_n, \beta_n)$ such that $k(\xi_n) \to \infty.$ Let us define $h_n = \frac1{2C_1} k(\xi_n)$. Then we see that for $s \le \xi_n$ we have the following inequality,
$$
k(s) \ge C_1 (s - \xi_n) + k(\xi_n)\ge \frac 12 k(\xi_n).
$$
As a result we reach
$$
\int_{\xi_n- h_n}^{\xi_n} k(s) \, ds \ge \int_{\xi_n- h_n}^{\xi_n} \frac 12 k(\xi_n)\, ds = \frac{h_n}2 k(\xi_n) = \frac{k(\xi_n)^2}{4C_1} >\pi 
$$
for sufficiently large $n$. 
Hence, we reach a contradiction with Lemma \ref{LMtoomuchangle}, and so there is $C_0>0$ such that $|k|\le C_0$.
\qed


\color{black}

\appendix

\end{document}